\documentclass[9pt, a5paper]{article}

\usepackage[english]{babel}
\usepackage[latin1]{inputenc}

\usepackage{ellipsis}		
\usepackage{xspace}			
\usepackage{exscale,relsize}		
\usepackage{ulem}			
\usepackage[hyphens]{url}
\usepackage{paralist}		
\usepackage{color,fancyvrb}		
\usepackage[force,full,safe]{textcomp} 	
\usepackage{float} 
\usepackage{fancyhdr}         
\usepackage{setspace}         

\usepackage{lmodern}
\usepackage{mathrsfs}
\usepackage{lipsum}
\usepackage{titlesec}

\usepackage[fleqn]{amsmath}		
\usepackage{amsthm}			
\usepackage{amsfonts}
\usepackage{amssymb}
\usepackage{bm,amscd} 

\usepackage{calc}
\usepackage{makeidx}
\usepackage{picinpar,graphicx}
\usepackage{tikz}
\usetikzlibrary{matrix,arrows,decorations.pathmorphing,positioning}
\usepackage[a4paper]{geometry}
\usepackage{makeidx}

\usepackage[colorlinks=true, linkcolor=blue, citecolor=brown]{hyperref}		

\newcounter{cntr}

\theoremstyle{plain}\newtheorem{thm}[cntr]{Theorem}
\theoremstyle{definition}\newtheorem{defn}[cntr]{Definition}
\theoremstyle{plain}\newtheorem{propn}[cntr]{Proposition}
\theoremstyle{plain}\newtheorem{cor}[cntr]{Corollary}
\theoremstyle{plain}\newtheorem{lem}[cntr]{Lemma}       
\theoremstyle{plain}
\theoremstyle{definition}
\theoremstyle{remark}\newtheorem{rem}[cntr]{Remark}
\theoremstyle{remark}

\newcommand\rarrow{\xrightarrow{\hspace*{0.7cm}}}
\titleformat{\section}{\normalfont\centering}{\thesection.}{12pt}{}
\title{Locus of restricted tangent bundles of Grassmannian to rational curves of given splitting type}
\author{Sayanta Mandal}

\newcommand\K{\mathbb{K}}
\newcommand\G{G(r,n)}
\renewcommand\P{\mathbb{P}}
\renewcommand\O{\mathcal{O}}
\renewcommand\S{\mathcal{S}}

\fancypagestyle{abcd}{
\fancyhf{}

\fancyfoot[L]{
}}

\begin{document}
\thispagestyle{abcd}

\begin{center}
\begin{Large}
On the loci of morphisms from $\P^1$ to $G(r,n)$ with fixed splitting type of the restricted universal sub-bundle or quotient bundle\\
\end{Large}
$\qquad$\\
\begin{large}
Sayanta Mandal\\
\end{large}
\end{center}

\paragraph{\normalfont\small\textsc{Abstract.}}
{\linespread{0.1}\footnotesize
Let $n\geq 4$, $2 \leq r \leq n-2$ and $e \geq 1$. We show that the intersection of the locus of degree $e$ morphisms from $\P^1$ to $G(r,n)$ with the restricted universal sub-bundles having a given splitting type and the locus of degree $e$ morphisms with the restricted universal quotient-bundle having a given splitting type is non-empty and generically transverse. As a consequence, we get that the locus of degree $e$ morphisms from $\P^1$ to $\G$ with the restricted tangent bundle having a given splitting type need not always be irreducible.
}

\section{\textsc{Introduction}}
Rational curves play a central role in the study of algebraic geometry of projective varieties. Let $X$ be a variety over an algebraically closed field $\K$, and let $C \subset X$ be a rational curve. The two bundles $T_X \vert_C$ and $N_{C/X}$ are especially important in understanding the deformations of $C$ in $X$ and understanding the geometry of the tangent space of smooth rational curves on $X$. These have been studied by Eisenbud and Van de Ven \cite{eis81}, \cite{eis82}, and by Ghione and Sacchiero \cite{ghione}, \cite{sac}, \cite{sac82} who characterized the possible splitting types of the normal bundle of rational curves in $\P^3$ and showed that the locus of rational curves in $\P^3$ with whose normal bundles have a specified splitting type is irreducible of the expected dimension. Ran \cite{ran} determined the splitting type of a generic genus-$0$ curve with one or two components in $\P^n$, as well as the way the bundle deforms locally with a general deformation of the curve.  More recently, Coskun and Riedl \cite{cos18}, \cite{cos19} showed that the locus of nondegenerate rational normal curves in $\mathbb{P}^n$ of fixed degree having a specified splitting type of the normal bundle can be reducible when $n \geq 5$.

In a similar vein, Verdier \cite{verdier} and Ramella \cite{ramella} showed that the locus of nondegenerate rational curves in $\P^n$ with a given splitting type of the restricted tangent bundle is irreducible of expected codimension. Str{\o}mme \cite{strom} examined a nice compactification of this locus as a certain Quot scheme and computed the Chow ring of this compactification. In this paper, we study the locus of degree $e$ morphism from $\P^1$ to the Grassmannian variety with a specified splitting type of the restricted tangent bundle.

Let $\G$ denote the Grassmannian variety of $r$-dimensional subspaces of a $n$-dimensional vector space. The Grassmannian has two special vector bundles, the universal sub-bundle $\mathcal{S}$ and the universal quotient bundle $\mathcal{Q}$ which fit together in an exact sequence 
\[ 0 \rarrow \mathcal{S} \rarrow \mathcal{O}_{\G}^{\oplus n} \rarrow \mathcal{Q} \rarrow 0 \]
Moreover, the tangent bundle $T_{\G}$ is isomorphic to $\mathcal{S}^* \otimes \mathcal{Q}$. We denote by $Mor_e(\P^1_{\K}, G(r,n))$ the scheme parameterizing degree $e$ morphisms from $\P^1_{\K}$ to $G(r,n)$. Let $M(b_\bullet)$ be the locus of morphisms $f$ in $Mor_e(\P^1_{\K}, \G)$ with $f^*(\mathcal{Q})$ having splitting type $b_1, \cdots, b_{n-r}$, and let $M'(a_\bullet)$ be the locus of morphism $f$ with $f^*(\mathcal{S^*})$ having splitting type $a_1, \cdots, a_r$. We first show that \newline
\textbf{Proposition} [Proposition \ref{codimensionlemma}]\textbf{.}\textit{ The loci $M(b_\bullet)$ and $M'(a_\bullet)$ are smooth of the expected codimension.}

This follows as a consequence of a Corollary due to Le Potier \cite{lep}[Corollary 15.4.3]. We then show that \newline
\textbf{Theorem} [Theorem \ref{maintheorem}]\textbf{.}\textit{ Let $n \geq 4$ and $2 \leq r \leq n-2$. The intersection of the loci $M(b_\bullet)$ and $M'(a_\bullet)$ is nonempty and generically transverse.}

The upshot of this theorem is the following: \newline
\textbf{Corollary} [Proposition \ref{proposition22}, Corollary \ref{corollary23}]\textbf{.}\textit{ The locus of morphisms $f$ in $Mor_e(\P^1, \G)$ with $f^*(T_{\G})$ having a specified splitting type has at least one irreducible component of expected codimension arising from a  possible splitting type of $f^*(\mathcal{S})$ and $f^*(\mathcal{Q})$. In particular, this locus need not always be irreducible.}

For example, (as a consequence of Corollary \ref{corollary23} and Lemma \ref{lemma24}) the locus of morphisms in $Mor_e(\P^1_{\K}, G(2,4))$ with restricted tangent bundle having splitting type $c_1,c_2, c_3, c_4$ with $c_1 \leq c_2 < c_3 \leq c_4$ has at least two irreducible components.

This is in sharp contrast with the result of Verdier \cite{verdier} and Ramella \cite{ramella} who have shown that the locus of morphisms $f$ in $Mor_e(\P^1_{\K}, \P^n_{\K})$ with the restricted twisted tangent bundle $f^*(T_{\P^n}(-1))$ having splitting type $a_1, \cdots, a_n$ with $a_1 \geq \cdots \geq a_n \geq 0$ and $a_1 + \cdots + a_n = e$ is a nonempty, smooth, irreducible subvariety.
\paragraph{Organization of the paper.} In section \ref{prelim}, we set up the notation and recall some facts on the restricted universal sub-bundle and restricted universal quotient bundle. In section \ref{nonemptylocussection}, we show that the intersection of the loci $M(b_\bullet)$ and $M'(a_\bullet)$ is nonempty. In section \ref{transverseintersectionsection}, we show that this intersection is generically transverse. In section \ref{section5}, we show that the locus of morphisms $f$ in $Mor_e(\P^1, \G)$ with $f^*(T_{\G})$ having a given splitting type need not always be irreducibe. Additionally, we analyze some special cases and give some examples.
\paragraph{Acknowledgements.} I am extremely grateful to my advisor Prof. Izzet Coskun for invaluable mathematical discussions, correspondences, and several helpful suggestions.
\section{\textsc{Preliminaries}}\label{prelim}
In this section, we set-up the notations and go over some preliminary results. Let $\K$ be an algebraically closed field of characteristic zero. Let $E$ be a vector bundle on $\P^1$ of rank $r$ and degree $e$. By Grothendieck's theorem, there are uniquely determined integers $a_1,\cdots, a_r$ with $a_1 \leq \cdots \leq a_r$ and $a_1 + \cdots + a_r = e$ such that $E$ is isomorphic to $\oplus_{i=1}^r \,\mathcal{O}_{\P^1}(a_i)$. We call this collection of integers the \textit{splitting type} of $E$. We say that $E$ is \textit{balanced} if $a_j - a_i \leq 1$ for all $1 \leq i,j \leq r$.

Let $n \geq 4$ and $2 \leq r \leq n-2$. We denote by $\G$ the Grassmannian variety of $r$-dimensional subspaces of a $n$-dimensional vector space. We denote by $M$ the scheme $Mor_e (\P^1, \G) $ parameterizing degree $e$ morphisms from $\P^1_\K$ to $\G$. We glean the following Lemma \ref{correspondence} from the universal property of Grassmannian
\begin{lem}\label{correspondence}
A degree $e$ morphism $\P^1 \rarrow \G$ corresponds uniquely to a vector bundle $E$ of rank $r$ and degree $e$ together with a surjection $\O^{\oplus n}_{\P^1} \rarrow E$.
\end{lem}
\begin{proof}
Given a morphism $\varphi : \P^1 \rarrow G(r,n)$, we take $E = \varphi^*(\mathcal{S}^*)$, where $\mathcal{S}$ is the universal sub-bundle, and we clearly have a surjection $v_\varphi : \mathcal{O}_{\P^1}^{\oplus n} \rarrow \varphi^*(\mathcal{S}^*)$. 

Conversely, given a surjection $ v : \mathcal{O}_{\P^1}^{\oplus n} \rarrow E$ where $E$ is a vector bundle of rank $r$ and degree $e$, let $s_1, \cdots, s_n$ form a basis for image of $H^0(\mathcal{O}_{\P^1}^{\oplus n})$ in $H^0(E)$, we have a morphism $\varphi_v : \P^1 \rarrow \P^{\binom{n}{r}}$  with co-ordinates given by $s_{i_1} \wedge \cdots \wedge s_{i_r}$ for $1 \leq i_1 < \cdots < i_r \leq n$, and we see that the image lies in $G(r,n)$ because the co-ordinates satisfy Pl{\"{u}}cker relations, and the resulting map has degree $e$ because $E$ has degree $e$.
\end{proof}

Subsequently, we can think of a morphism from $\P^1$ to $G(r,n)$ as an element of the quot scheme $Quot^{r,e}_{\O^{\oplus n}_{\P^1}/\P^1/\K}$, which is an irreducible, rational, nonsingular, projective variety of dimension $r(n-r) + ne$ \cite[Theorem 2.1]{strom}. In particular, we can think of $M$ as a subscheme of $Quot^{r,e}_{\O^{\oplus n}_{\P^1}/\P^1/\K}$.

\begin{lem}
$M$ is an open subscheme of the quot scheme $Quot^{r,e}_{\O^{\oplus n}_{\P^1}/\P^1/\K}$. Therefore, $M$ is a smooth quasi-projective variety of dimension $r(n-r) + ne$.
\end{lem}
\begin{proof}
Note that any coherent sheaf $E$ on $\P^1$ has a unique decomposition $E = E' \oplus T$, where $E'$ is locally free and $T$ is torsion. Given any $1 \leq i \leq e$, let $X_i$ be the image of the map 
\[ Quot^{r,e-i}_{\O^{\oplus n}_{\P^1}/\P^1/\K} \times \P^1 \times \cdots_{(i \text{ times })} \cdots \times \P^1 \rarrow Quot^{r,e}_{\O^{\oplus n}_{\P^1}/\P^1/\K}\]
which sends $(E', x_1, \cdots, x_i)$ to $E' \oplus T$ where $T$ is the structure sheaf of the closed subscheme of $\P^1$ defined by $\{ x_1, \cdots, x_i \}$. We see that $X_i$ is closed and irreducible because it is the image of a proper irreducible variety. We have 
\[ dim(X_i) \leq r(n-r) + n(e-i) + i < r(n-r) + ne \]
Since every coherent sheaf $E$ of rank $r$ and degree $e$ which is not locally free lies in some $X_i$, we conclude that $M$ is the complement of the union of the $X_i$'s for $1 \leq i \leq e$.
\end{proof}

We have a canonical map
\[ \Phi : M \times \P^1 \rarrow \G \] which sends a pair $(f,x)$ to $f(x)$. Let $\S$ denote the universal bundle over $\G$.

\begin{lem}\label{claim1}
The family of vector bundles parametrized by $M$ via $\Phi^*(\S^*) \rarrow M \times \P^1$ is a complete family.
\end{lem}
\begin{proof}
Let $E = \O_{\P^1}(a_1) \oplus \cdots \oplus \O_{\P^1}(a_r)$ and $K = \O_{\P^1}(-b_1) \oplus \cdots \oplus \O_{\P^1}(-b_{n-r})$, where $deg(E) = - deg(K) = e$, and consider the exact sequence 
\[ 0 \rarrow K \rarrow \O^{\oplus n}_{\P^1} \rarrow E \rarrow 0 \]
We first observe that if $f$ is the morphism corresponding to $\O_{\P^1}^{\oplus n} \rarrow E$, then $\Phi^*(\S^*)\vert_f = E$. We look at the following commutative diagram
\begin{equation*}
\begin{tikzpicture}[>=angle 90]
	\matrix(a)[matrix of math nodes, row sep=4em, column sep=4.5em, text height=1.5ex, text depth=0.25ex]
	{T_f (M) & Ext^1(\Phi^*(\S^*)\vert_f, \Phi^*(\S^*)\vert_f) \\
	Hom(K,E) & Ext^1(E,E)\\};	

	\path[->] (a-1-1) edge node[above]{} (a-1-2);
	\path[->] (a-1-1) edge node[left]{} (a-2-1);
	\path[->] (a-2-1) edge node[below]{} (a-2-2);
	\path[->] (a-1-2) edge node[right]{} (a-2-2);
\end{tikzpicture}
\end{equation*}
where the vertical maps are isomorphisms, the top horizontal map is the Kodaira-Spencer map, and the bottom horizontal map is obtained by applying $Hom( \bullet , E)$ to the exact sequence 
\[ 0 \rarrow K \rarrow \O_{\P^1}^{\oplus n} \rarrow E \rarrow 0 \]
Since the next term in the long exact sequence is $Ext^1( \O_{\P^1}^{\oplus n} , E) = H^1(E)^{\oplus n} = 0$, the bottom horizontal map is surjective. Hence, the Kodaira-Spencer map is surjective, and so the family is complete.
\end{proof}

We will now use the following corollary due to Le Potier to conclude that the locus of quotient vector bundles in M of given splitting type has expected codimension.

\begin{propn}[\cite{lep}, Cor 15.4.3]\label{lepotier1}
Let X be a smooth projective curve of genus g. Let $E_s$ be a complete family of vector bundles of rank $r$ and degree $d$ parametrized by a smooth variety S. For integers $l, r_i >0 $ and $d_i$, set \[ \mu_i = \frac{d_i}{r_i} \]
The points $s \in S$ such that the Harder-Narasimhan filtration (if it exists) has length $l$ and such that the Harder-Narasimhan grading $gr_i(E_s)$ of $E_s$ has rank $r_i $ and degree $d_i$, for $i = 1, \cdots, l$, form a locally closed smooth subvariety of codimension 
\[ \sum_{i<j} r_i r_j( \mu_i - \mu_j + g - 1) \]
\end{propn}
Observe that when $g=0$, we have $E_s = \oplus_{i=1}^r \,\mathcal{O}_{\P^1}(a_i)$ for some integers $a_1,\cdots, a_r$, and so, 
\begin{equation}\label{lepotiereqn}
\sum_{i<j} r_i r_j (\mu_i - \mu_j -1) = ext^1(E_s,E_s) = \sum_{i,j} \max \left\lbrace a_i - a_j - 1, 0 \right\rbrace 
\end{equation}

Now we fix two collection of non-negative integers $a_1 \geq \cdots \geq a_r \geq 0$ and $0 \leq b_1 \leq \cdots \leq b_{n-r} $ such that $a_1 + \cdots + a_r = b_1 + \cdots + b_{n-r} = e >0$. Let $M(b_\bullet)$ be the locus of morphisms in $M$ with the restricted universal quotient bundle being isomorphic to $\O_{\P^1}(b_1) \oplus \cdots \oplus \O_{\P^1}(b_{n-r})$, and let $M'(a_\bullet) $ be the locus of morphisms in $M$ with the restricted universal sub-bundle being isomorphic to $\O_{\P^1}(-a_1) \oplus \cdots \oplus \O_{\P^1}(-a_{r})$. Our goal is to show that $M(b_\bullet) \cap M'(a_\bullet)$ is nonempty and generically smooth of expected codimension 
\[\sum_{1 \leq i,j \leq r} \max \left\{ a_i -a_j - 1, 0 \right\} + \sum_{1 \leq i,j \leq n-r} \max \left\{ b_i - b_j - 1, 0 \right\} \]
We see that 
\begin{propn}\label{codimensionlemma}
The locus $M(b_\bullet)$ is smooth of codimension 
\[ \sum_{i,j} \max \left\lbrace b_i - b_j - 1, 0 \right\rbrace \]
Similarly, $M'(a_\bullet)$  is smooth of codimension
\[ \sum_{i,j} \max \left\lbrace a_i - a_j - 1, 0 \right\rbrace \]
\end{propn}
\begin{proof}
The first part of the Lemma follows from Lemma \ref{claim1}, Proposition \ref{lepotier1}, and equation \ref{lepotiereqn}.

To conclude the second part, we note that the canonical map 
\[ Quot^r_{\O_{\P^1}^{\oplus n}/\P^1/\K} \rarrow Quot^{n-r}_{\O_{\P^1}^{\oplus n}/\P^1/\K} \]
which sends $[\O_{\P^1}^{\oplus n} \rarrow E]$ to $[\O_{\P^1}^{\oplus n} \rarrow \mathcal{K}^*]$, where $\mathcal{K}$ is the kernel of the map $\O_{\P^1}^{\oplus n} \rarrow E $, induces an isomorphism between $Mor_{e}(\P^1, G(r,n))$ and $Mor_e(\P^1, G(n-r,n))$. Hence, the second part of the Lemma follows from the first part.
\end{proof}

Therefore, we need to show that the intersection of $M(b_\bullet)$ and $M'(a_\bullet)$ is nonempty, and we need to find a point in $M(b_\bullet) \cap M'(a_\bullet)$ where the intersection is transverse. We show these in section \ref{nonemptylocussection} and \ref{transverseintersectionsection}.

\begin{defn}\label{polygon}
Given a collection of non-negative integers $a_1, \cdots, a_l$, we define its \textit{polygonal line} to be 
\[ \mathfrak{P}(a_1, \cdots, a_l) = (a'_1, a'_1 + a'_2, \cdots, a'_1 + \cdots + a'_l) \]
where $a'_1, \cdots , a'_l$ is a rearrangement of the $a_i$'s such that $a'_1 \geq \cdots \geq a'_l$. Additionally, given another such collection $b_1, \cdots, b_l$ with rearrangement $b'_1 \geq \cdots \geq b'_l$, we define inequality 
\[ \mathfrak{P}(b_\bullet) \geq \mathfrak{P}(a_\bullet) \quad \text{if} \quad \sum_{j=1}^i b'_j \geq \sum_{j=1}^i a'_j, \quad \text{for all } 1 \leq i \leq l \]
\end{defn}

It follows as a consequence of Proposition 1.2 due to Ramella \cite{ramella}
\begin{propn}\label{mclosure}Given two collection of non-negative integers $0 \leq b_1 \leq \cdots \leq b_{n-r}$ and $0 \leq b'_1 \leq \cdots \leq b'_{n-r}$ with $b_1 + \cdots + b_{n-r} = b'_1 + \cdots + b'_{n-r} = e$. We have 
\[ \overline{M(b'_\bullet)} \supset M(b_\bullet) \quad \text{iff} \quad \mathfrak{P}(b'_\bullet) \leq \mathfrak{P}(b_\bullet) \]

Similar result holds for $M'(a_\bullet)$.
\end{propn}
Since $M$ is stratified by $M(b_\bullet)$ for all possible $0 \leq b_1 \leq \cdots \leq b_{n-r}$ with $b_1 + \cdots + b_{n-r} = e$, and by $M'(a_\bullet)$ for all possible $a_1 \geq \cdots \geq a_r \geq 0$ with $a_1 + \cdots + a_r = e$, Proposition \ref{mclosure} yields the following Corollary.
\begin{cor}\label{mclosure2}
The closure of the locus $M(b_\bullet)$ in $Mor_e(\mathbb{P}^1, G(r,n))$ is 
\[ \overline{M(b_\bullet)} = \bigcup_{\substack{0 \leq b'_1 \leq \cdots \leq b'_{n-r} \\ b_1 + \cdots + b_{n-r} = e \\ \mathfrak{P}(b'_\bullet) \geq \mathfrak{P}(b_\bullet)}} M(b'_\bullet) \]
Similarly, we have 
\[ \overline{M'(a_\bullet)} = \bigcup_{\substack{a'_1 \geq \cdots \geq a'_r \geq 0 \\ a'_1 + \cdots + a'_r = e \\ \mathfrak{P}(a'_\bullet) \geq \mathfrak{P}(a_\bullet)}} M'(a'_\bullet) \]
\end{cor}

\section{\textsc{The intersection locus is nonempty}}\label{nonemptylocussection}
In this section we show that the intersection of the locus of degree $e$ morphisms from $\P^1$ to $G(r,n)$ with the restricted universal sub-bundle having given splitting type and the locus of degree $e$ morphisms with restricted universal quotient bundle having given splitting type is non-empty. In particular, we want to show that given two sequences of non-negative integers $a_1 \geq \cdots \geq a_r \geq 0$ and $0 \leq b_1 \leq \cdots \leq b_{n-r} $ such that $a_1 + \cdots + a_r = b_1 + \cdots + b_{n-r} = e >0$, there exits an exact sequence of vector bundles 
\[ 0 \rarrow \O_{\P^1}(-b_1) \oplus \cdots \oplus \O_{\P^1}(-b_{n-r}) \stackrel{u}{\rarrow} \O_{\P^1}^{\oplus n} \stackrel{v}{\rarrow} \O_{\P^1}(a_1) \oplus \cdots \oplus \O_{\P^1}(a_r) \rarrow 0 \]
By dualizing the sequence if necessary, we may assume without loss of generality that $(n-r) \leq r$.



Before doing the general case, we would like to do the case $r = n-r = 2$. We have $a_1 \geq a_2$, $b_1 \leq b_2$ and $a_1 + a_2 = b_1 + b_2 = e$. 
\begin{propn}\label{baby}
There exists an exact sequence 
\[ 0 \rarrow \O(-b_1) \oplus \O(-b_2) \stackrel{u}{\rarrow} \O^{\oplus 4} \stackrel{v}{\rarrow} \O(a_1) \oplus \O(a_2) \rarrow 0 \]
\end{propn}
\begin{proof} Note that we must have $a_1 \geq b_1$, otherwise
\[b_1 + b_2 \geq 2 b_1 >2a_1 \geq a_1 + a_2\]
which is a contradiction.
We define 
\[ 
v = \begin{pmatrix}
x^{a_1} & y^{a_1} & 0 & x^{a_1 - b_1}y^{b_1}\\
0 & x^{a_2} & y^{a_2} & 0
\end{pmatrix} \qquad \qquad
u = \begin{pmatrix}
-y^{b_1} & 0 \\
0 & x^{a_1 - b_1}y^{a_2} \\
0 & -x^{b_2}\\
x^{b_1} & -y^{b_2}
\end{pmatrix}
\]
where $x$ and $y$ denote the co-ordinate functions of $\P^1$. The minor corresponding to the first two columns of $v$ is $x^{a_1 + a_2}$ and the minor corresponding to the second and third column of $v$ is $y^{a_1 + a_2}$. Since these two monomials do not vanish simultaneously on $\P^1$, we conclude that $v$ is surjective. 

Similarly, by looking at the minor corresponding to first and fourth row, and the minor corresponding to third and fourth row, we conclude that $u$ is injective. 

Finally, one can check that $v \circ u = 0$.
\end{proof}

Now we discuss the general case when $(n-r) \leq r$. We define 
\begin{defn}\label{definition7}
\[ A(j) = \begin{cases} 0, \qquad \qquad \qquad \quad \text{ if } j \leq 0 \\
a_1 + \cdots + a_j, \qquad \text{ if } 1 \leq j \leq r \\
a_1 + \cdots + a_r, \qquad \text{ if } j \geq r
\end{cases}
B(i) = \begin{cases} 0, \qquad \qquad \qquad \quad \text{ if } i \leq 0 \\
b_1 + \cdots + b_i, \qquad \text{ if } 1 \leq i \leq n-r \\
b_1 + \cdots + b_{n-r}, \quad \text{ if } i \geq n-r
\end{cases}
\]
\end{defn}
To describe the matrices, we need to use the following lemma.
\begin{lem}\label{order}
Let $a_1 \geq \cdots \geq a_r \geq 0$ and $0 \leq b_1 \leq \cdots \leq b_{n-r}$ be two sequence of non-negative integers with $(n-r) \leq r$ and $ A(r) = B(n-r)$. Then for all $0 \leq l \leq (n-r)$, we have $A(2r -n + l) \geq B(l)$.
\end{lem}
\begin{proof}
Let $s(l) = A(2r -n +l) - B(l) $ for any $0 \leq l \leq (n-r)$. Clearly, $s(0) \geq 0$. 
Let $1 \leq l_0 < n-r$ be the least integer such that $s(l_0-1) \geq 0$ and $s(l_0)<0$. 

Since $s(l_0) = s(l_0-1) + a_{2r -n +l_0} - b_{l_0}$, we must have $a_{2r -n +l_0} - b_{l_0} <0$. This in turn implies that 
\[ a_r \leq \cdots \leq a_{2r -n +l_0+1} \leq a_{2r -n + l_0}  < b_{l_0} \leq b_{l_0+1} \leq \cdots \leq b_{n-r} \]
which gives 
\[s(n-r)  = s(l_0) + (a_{2r -n + l_0 +1} - b_{l_0 + 1}) + \cdots + (a_r - b_{n-r})\leq  s(l_0) < 0\]
But we know $s(n-r) = 0$, thus we have a contradiction.
\end{proof}

The description of the matrices depend on how the $A(j)$'s and $B(i)$'s are ordered. For example, let $r = n-r = 5$ and let's assume the following order 
\[ B(1) < B(2) < A(1) < A(2) < B(3) < A(3) < B(4) < A(4) < B(5) = A(5) \]
For ease of notation, let us define $s_{j,i} = A(j) - B(i)$ for any given integers $i,j$. Let $x$ and $y$ denote the co-ordinate functions of $\P^1$. Then the first matrix $v$ is given as follows : 
{\footnotesize
\begin{equation}
\begin{pmatrix}
x^{a_1} & y^{a_1} & 0 & 0 & 0 & 0 & x^{s_{1,1}}y^{-s_{0,1}} & x^{s_{1,2}}y^{-s_{0,2}} & 0 & 0 \\
0 & x^{a_2} & y^{a_2} & 0 & 0 & 0 & 0 & 0 & 0 & 0 \\
0 & 0 & x^{a_3} & y^{a_3} & 0 & 0 & 0 & 0 & x^{s_{3,3}}y^{-s_{2,3}} & 0\\
0 & 0 & 0 & x^{a_4} & y^{a_4} & 0 & 0 & 0 & 0 & x^{s_{4,4}}y^{-s_{3,4}} \\
0 & 0 & 0 & 0 & x^{a_5} & y^{a_5} & 0 & 0 & 0 & 0\\
\end{pmatrix}
\end{equation}
}
The second matrix $u$ is given as follows : 
\begin{equation}
\begin{pmatrix}
-y^{b_1} & 0 & 0 & 0 & 0\\
0 & 0 & x^{s_{1,2}}y^{-s_{1,3}} & 0 &0 \\
0 & 0 & -x^{s_{2,2}}y^{-s_{2,3}} & 0 & 0 \\
0 & 0 & 0 & x^{s_{3,3}}y^{-s_{3,4}} & 0 \\
0 & 0 & 0 & 0 & x^{s_{4,4}}y^{-s_{4,5}}\\
0 & 0 & 0 & 0 & -x^{b_5}\\
x^{b_1} & -y^{b_2} & 0 & 0 & 0 \\
0 & x^{b_2} & -y^{b_3} & 0 & 0\\
0 & 0 & x^{b_3} & - y^{b_4} & 0\\
0 & 0 & 0 & -x^{b_4} & -y^{b_5}\\
\end{pmatrix}
\end{equation}
It is easy to see that the $v$ is surjective, $u$ is injective, and $v \circ u = 0$. 

We now proceed to define the matrices $v$ and $u$ in general. We define two increasing sequences of non-negative integers $ \lbrace i_l \rbrace_{l \geq 0} $ and $ \lbrace j_l \rbrace_{l\geq 0} $ recursively in the following manner:

We define $i_0 = 0$, and $j_0$ to be the largest non-negative integer such that $j_0 \leq r$ and  $A(j_0) \leq B(1)$. For each $ l \geq 1$, we define $i_l$ to be the largest non-negative integer such that $i_l \leq n-r$ and $B(i_l) \leq A(j_{l-1}+1)$ and $j_l$ to be the largest non-negative integer such that $j_l \leq r$ and $A(j_l) \leq B(i_l + 1)$. It follows that for $l \gg 0$, we have $j_l = r$ and $i_l = n-r$. We define $\alpha$ to be the least positive integer such that $j_{\alpha+1} = r$. It follows from Lemma \autoref{order} that in general, there are two possible orderings:

if $a_1 > b_1$, we see that $i_0 = j_0 = 0$ and we have:
\begin{align*}
& B(1) \leq \cdots \leq B(i_1) \leq A(1) \leq \cdots \leq A(j_1) \leq \\
& B(i_1 + 1) \leq \cdots \leq B(i_2) \leq A(j_1 + 1) \leq \cdots \leq A(j_2) \leq \cdots \leq \\ 
& B(i_\alpha + 1) \leq \cdots \leq B(n-r-1) \leq A(j_\alpha +1) \leq \cdots \leq A(r-1) \leq A(r) = B(n-r)
\end{align*}

if $a_1 \leq b_1$, we have:
\begin{align*}
& A(1) \leq \cdots \leq A(j_0) \leq B(1) \leq \cdots \leq B(i_1) \leq A(j_0 +1) \leq \cdots \leq A(j_1) \leq \\
& B(i_1 + 1) \leq \cdots \leq B(i_2) \leq A(j_1 + 1) \leq \cdots \leq A(j_2) \leq \cdots \leq \\ 
& B(i_\alpha + 1) \leq \cdots \leq B(n-r-1) \leq A(j_\alpha +1) \leq \cdots \leq A(r-1) \leq A(r) = B(n-r)
\end{align*}

We define the first matrix $v_{r \times n}$ as follows: 
we have a $r \times (r+1)$ block matrix and a $r \times (n-r-1)$ block matrix comprising the matrix $v_{r \times n}$. The $r \times (r+1)$ block matrix has diagonal and super-diagonal entries defined as follows:
\begin{align*}
 v_{i,i} = x^{a_i}, \text{ for } i = 1, \cdots, r; \qquad v_{i, i+1} = y^{a_i} , \text{ for } i = 1, \cdots, r; 
\end{align*}
All the remaining entries of this block are zero. The $r \times (n-r-1)$ block has non-zero entries only in rows $j_0+1, j_1 + 1, \cdots, j_\alpha +1$, and all other rows have all zero entries. 
For $0 \leq l \leq \alpha -1$, the row $j_l +1$ have non-zero entries in columns $r+2+i_l $ upto $r+1+i_{l+1}$ and zero entries for all other columns. The non-zero entries are: 
\[  v_{j_l + 1, r + 2 + i_l} = x^{A(j_l + 1) - B(i_l + 1)}y^{B(i_l + 1) - A(j_l)}, \cdots , v_{j_l + 1, r + 1 + i_{l+1}} = x^{A(j_l + 1) - B(i_{l+1})}y^{B(i_{l+1}) - A(j_l)}
\]
The row $j_\alpha +1$ has non-zero entries in columns $r+2+i_\alpha $ upto $n$, and zero entries in all other columns. The non-zero entries are: 
\[
v_{j_\alpha + 1, r+2+ i_\alpha } = x^{A(j_\alpha + 1) - B(i_\alpha + 1)}y^{B(i_\alpha + 1) - A(j_\alpha)}, \cdots, v_{j_\alpha + 1, n} = x^{A(j_\alpha + 1) - B(n-r-1)}y^{B(n-r-1) - A(j_\alpha)}
\]

We now proceed to define the second matrix $u_{n \times (n-r)}$. The matrix $u$ comprises of three blocks, a $(j_0 +1) \times (n-r)$ block $u1$ consisting of the first $j_0 +1$ rows of $u$, a $(r-j_0-1) \times (n-r)$ block $u2$ consisting of rows $j_0 + 2$ upto $r$ of $u$, and a $(n-r) \times (n-r)$ block $u3$ consisting of rows $r+1$ upto $n$ of $u$.

The matrix $u1$ has non-zero entries in the first column and zero entries in all remaining columns. The non-zero entries are: 
\[ u_{1,1} = - y^{b_1}, u_{2,1} = (-1)^2 x^{A(1)}y^{B(1) - A(1)}, \cdots, u_{j_0 +1,1} = (-1)^{j_0 + 1}x^{A(j_0)}y^{B(1) - A(j_0)} \]
The matrix $u2$ has non-zero entries in columns $i_1+1, i_2+1, \cdots, i_\alpha +1$ and $(n-r)$, and zero entries in all other columns. For any $1 \leq l \leq \alpha$, the column $i_l +1$ has non-zero entries in rows $j_{l-1}+2, \cdots, j_l +1$ and zero entries in all other rows. The non-zero entries are: 
\begin{align*} 
u_{j_{l-1} + 2, i_l + 1} &= (-1)^{j_{l-1} + 2 - (j_{l-1} + 2)}x^{A(j_{l-1} + 1) - B(i_l)}y^{B(i_l + 1) - A(j_{l-1} + 1)}\\
u_{j_{l-1} + 3, i_l + 1} &= (-1)^{j_{l-1} + 3 - (j_{l-1} + 2)}x^{A(j_{l-1} + 2) - B(i_l)}y^{B(i_l + 1) - A(j_{l-1} + 2)}\\
&\vdots\\
u_{j_l + 1, i_l + 1} &= (-1)^{j_l + 1 - (j_{l-1} + 2)}x^{A(j_l) - B(i_l)}y^{B(i_l + 1) - A(j_l)}
\end{align*}
The $(n-r)$th column has non-zero entries in rows $j_\alpha +2$ upto $r$, and has zero entries in all other rows. The non-zero entries are:
\begin{align*}
&u_{j_\alpha +2,n-r} = (-1)^{j_\alpha+2 - (j_\alpha + 2)}x^{A(j_{\alpha}+1) - B(n-r-1)}y^{B(n-r) - A(j_\alpha +1)}, \cdots \\
& \qquad \qquad \qquad \cdots u_{r,n-r} = (-1)^{r - (j_\alpha +2)}x^{A(r-1)-B(n-r-1)}y^{B(n-r)-A(r-1)}
\end{align*}
The non-zero entries of matrix $u3$ are along the diagonal, the sub-diagonal, and in the $(n-r)$th column. The diagonal entries are:
\begin{align*} u_{r+i, i} = \begin{cases} 0 , &\text{ if } i=1 \\- y^{b_{i}}, &\text{ if } 2 \leq i \leq (n-r) \end{cases}
\end{align*}
The sub-diagonal entries are:
\[ u_{r+1+i, i} = (-1)^{\beta_i} x^{b_i}, \text{ for } i = 1, \cdots, n-r-1 \]
where $\beta_i$ denotes the number of $A(j)$'s lying strictly in between $B(i)$ and $B(i-1)$. We also have $u_{r+1, n-r} = (-1)^{\beta_{n-r}x^{b_{n-r}}}$. All other entries are zero.
\begin{propn}\label{nonempty}
The matrix $v$ is surjective, $u$ is injective, and $v\circ u = 0$. In particular, we have an exact sequence 
\[ 0 \rarrow \oplus_{j=1}^{n-r} \mathcal{O}_{\P^1}(-b_j) \stackrel{u}{\rarrow} \mathcal{O}_{\P^1}^{\oplus n} \stackrel{v}{\rarrow} \oplus_{i=1}^r \mathcal{O}_{\P^1} (a_i) \rarrow 0 \]
\end{propn}
\begin{proof}
It follows from the definition of $v$ that every entry in the $i$th row of $v$ is either zero or a monomial of degree $a_i$ in $x$ and $y$, where $x$ and $y$ are the co-ordinate functions of $\P^1$. Hence, $v$ defines a morphism from $\mathcal{O}_{\P^1}^{\oplus n} $ to $\oplus_{i=1}^r \,\mathcal{O}_{\P^1}(a_i)$. 

Similarly, it follows from definition of $u$ that every entry in the $j$th column of $u$ is either zero or a monomial of degree $b_j$ in $x$ and $y$, a posteriori, defining a morphism from $\oplus_{j=1}^{n-r} \,\mathcal{O}_{\P^1}(-b_j)$ to $\mathcal{O}_{\P^1}^{\oplus n}$.

To show $v$ is surjective, we look at two $r \times r$ minors of $v$, the first one consisting of the first $r$ columns and the second one consisting of columns $2, \cdots , (r+1)$
\[ (v_{p,q})_{1\leq p ,q\leq r} \text{ and } (v_{p,q})_{1 \leq p \leq r,2 \leq q \leq r+1 } \]
The determinant of first one is $x^{a_1 + \cdots + a_r}$ and second one is $y^{a_1 + \cdots + a_r}$, which do not vanish simultaneously at any point of $\P^1$.

Similarly, to show $u$ is injective, we look at two $(n-r) \times (n-r)$ minors, the first one consisting of rows $(r+1), \cdots, n$ and the second one consisting of row $1$ and rows $r+2, \cdots , n$
\[ (u_{p,q})_{r+1\leq p \leq n, 1\leq q\leq n-r} \text{ and } (u_{p,q})_{p=1, r+2 \leq p \leq n, 1\leq q \leq n-r} \]
The determinant of first one is $(-1)^{\beta_1 + \cdots + \beta_{n-r} + n-r+1} x^{b_1 + \cdots + b_{n-r}}$ and for the second one is $(-1)^{n-r} y^{b_1 + \cdots + b_{n-r}}$, which do not vanish simultaneously at any point of $\P^1$.

Before we begin proof of third part, we would like to explicitly write down the $\beta_i 's$ used in the description of the matrix $u$. Recall that $\beta_i$ is the number of $A(j)'s$ lying strictly in between $B(i)$ and $B(i-1)$. Thus, when we are in first case where $a_1 > b_1$, we have
\[
\beta_i = \begin{cases}
j_l - j_{l-1},  &\text{ if } i = i_l + 1, \, 1 \leq l \leq \alpha \\
r - j_\alpha - 1 ,  &\text{ if } i = n-r\\
0,  &\text{ otherwise}
\end{cases}
\]

and when we are in second case where $a_1 \leq b_1$, we have
\[
\beta_i = \begin{cases}
j_0, & \text{ if } i=1\\
j_l - j_{l-1},  &\text{ if } i = i_l + 1, \, 1 \leq l \leq \alpha \\
r - j_\alpha - 1 ,  &\text{ if } i = n-r\\
0,  &\text{ otherwise}
\end{cases}
\]
Let $v_p$ denote the $p$th row of $v$, and $u_q$ denote the $q$th column of $u$. Our goal is to show that for any $1 \leq q \leq (n-r)$, we have $v_p \cdot u_q = 0$ for every $1 \leq p \leq r$, and hence we can conclude that $v \cdot u =0$.

We first analyze the case when $a_1 >b_1$. Now $u_1$ has nonzero entry in the first and $(r+2)$th row, and $v_1$ is the only row in $v$ with nonzero entries in the respective columns. We see 
\[ v_1 \cdot u_1 = x^{a_1} \cdot (-y^{b_1}) + x^{A(1) - B(1)}y^{B(1)}\cdot x^{b_1} = -x^{a_1}y^{b_1} + x^{a_1}y^{b_1} = 0 \]
Thus, for any $1 \leq p \leq r$ we have $v_p \cdot u_1 = 0$.

For $2 \leq q \leq n-r-1$ and $q \neq i_1 + 1, \cdots , i_\alpha + 1$, $u_q$ has nonzero entry in $(r+q)$th and $(r+1 + q)$th row. By construction, $u_{r+q,q} = -y^{b_q}$ and $u_{r+q+1,q} = x^{b_q}$.
Let $A(j_l) \leq B(q) \leq A(j_l + 1)$ for some $0 \leq l \leq \alpha$ as per our chosen ordering, then the $(r+q)$th and $(r+1+q)$th columns of $v$ have nonzero entry only in row $j_l + 1$, and the entries are $v_{j_l + 1, r+q} = x^{A(j_l + 1) - B(q-1)}y^{B(q-1) - A(j_l)}$ and $v_{j_l + 1, r+1+q} = x^{A(j_l + 1) - B(q)}y^{B(q) - A(j_l)}$. Thus, 
\begin{align*} 
v_{j_l + 1} \cdot u_q &= x^{A(j_l + 1) - B(q-1)}y^{B(q-1) - A(j_l)} \cdot (-y^{b_q}) + x^{A(j_l + 1) - B(q)}y^{B(q) - A(j_l)} \cdot x^{b_q} \\
& = - x^{A(j_l + 1) - B(q-1)}y^{B(q) - A(j_l)} + x^{A(j_l + 1) - B(q-1)}y^{B(q) - A(j_l)} = 0 
\end{align*}
Hence, for any $1 \leq p \leq r$ and $2 \leq q \leq n-r-1$, $q \neq i_1+1, \cdots, i_\alpha +1$, we have $v_p \cdot u_q = 0$.

Suppose $q = i_l + 1$ for some $1 \leq l \leq \alpha$, by construction $u_{i_l + 1} $ has nonzero entries in rows $j_{l-1} + 2 , \cdots, j_l + 1$, $r+1+i_l$ and $r+1 + (i_l + 1)$. By our chosen ordering, we have 
\[ B(i_l) \leq A(j_{l-1} + 1) \leq \cdots \leq A(j_l) \leq B(i_l + 1) \leq A(j_l + 1) \]
Clearly the rows $j_{l-1} +1, \cdots, j_l +1$ of $v$ are the only ones in which there is a nonzero entry in the columns corresponding to the aforementioned rows of $u$. We have 
\begin{align*}
 v_{j_{l-1}+1} \cdot u_{i_l + 1} & = y^{a_{j_{l-1}+1}} \cdot x^{A(j_{l-1}+1) - B(i_l)}y^{B(i_l + 1) - A(j_{l-1}+1)}  \\
 & \qquad \qquad + x^{A(j_{l-1}+1) - B(i_l)}y^{B(i_l) - A(j_{l-1})} \cdot (-y^{b_l}) \\
 &  = x^{A(j_{l-1}+1) - B(i_l)}y^{B(i_l + 1) - A(j_{l-1})} - x^{A(j_{l-1}+1) - B(i_l)}y^{B(i_l + 1) - A(j_{l-1})} =0 \qquad \qquad \;\; 
\end{align*}
\begin{align*}
 v_{j_l +1} \cdot u_{i_l + 1} & = x^{a_{j_l +1}} \cdot (-1)^{j_l + 1 - (j_{l-1}+2)}x^{A(j_l) - B(i_l)}y^{B(i_l + 1) - A(j_l)} \\
 & \qquad \qquad + x^{A(j_l +1) - B(i_l + 1)}y^{B(i_l + 1) - A(j_l)} \cdot (-1)^{j_l - j_{l-1}}x^{b_l}\\
 & = (-1)^{j_l -j_{l-1}-1}x^{A(j_l + 1) - B(i_l)}y^{B(i_l + 1) - A(j_l)} \\
 & \qquad \qquad + (-1)^{j_l - j_{l-1}}x^{A(j_l + 1) - B(i_l)}y^{B(i_l + 1) - A(j_l)} = 0\\
\end{align*}
For $c = 2,3 , \cdots , j_l - j_{l-1}$, we have
\begin{align*}
v_{j_{l-1} + c} \cdot u_{i_l + 1} & = x^{a_{j_{l-1} + c}}  \cdot (-1)^{j_{l-1}+c - (j_{l-1}+2)}x^{A(j_{l-1}+c-1) - B(i_l)}y^{B(i_l + 1) - A(j_{l-1}+c-1)} \\
& \qquad \qquad + y^{a_{j_{l-1}+c}} \cdot (-1)^{j_{l-1}+c+1 - (j_{l-1}+2)}x^{A(j_{l-1}+c) - B(i_l)}y^{B(i_l +1) - A(j_{l-1} + c)}\\
& = (-1)^{c-2} x^{A(j_{l-1}+c) - B(i_l)}y^{B(i_l + 1) - A(j_{l-1}+c-1)} \\
& \qquad \qquad + (-1)^{c-1}x^{A(j_{l-1}+c) - B(i_l)}y^{B(i_l + 1) - A(j_{l-1}+c-1)} = 0
\end{align*}
Hence, for any $1 \leq p \leq r$ and $q = i_l + 1$ for $1 \leq l \leq \alpha$, we have $v_p \cdot u_q = 0$.

By construction, $u_{n-r}$ has a non-zero entry in rows $j_\alpha + 2, \cdots, r, r+1$ and $n$. The rows $v_p$ of $v$ such that there is a non-zero entry in any of the columns corresponding to non-zero rows of $u_{n-r}$ are $p = j_\alpha + 1, \cdots, r$. We have
\begin{align*}
v_{j_\alpha +1} \cdot u_{n-r} & = y^{a_{j_\alpha + 1}} \cdot (-1)^{j_\alpha +2 - (j_\alpha + 2)}x^{A(j_\alpha + 1) - B(n-r-1)}y^{B(n-r) - A(j_\alpha + 1)} \\
& \qquad \qquad + x^{A(j_\alpha + 1) - B(n-r-1)}y^{B(n-r-1) - A(j_\alpha)} \cdot (-y^{b_{n-r}})\\
& = x^{A(j_\alpha +1) - B(n-r-1)}y^{B(n-r) - A(j_\alpha)} - x^{A(j_\alpha + 1) - B(n-r-1)}y^{B(n-r) - A(j_\alpha)} = 0
\end{align*}
Similarly,
\begin{align*}
v_r \cdot u_{n-r} = x^{a_r} \cdot (-1)^{r - (j_\alpha +2)}x^{A(r-1) - B(n-r-1)}y^{B(n-r) - A(r-1)} + y^{a_r} \cdot (-1)^{\beta_{n-r}}x^{b_{n-r}}
\end{align*}
Recall that $\beta_{n-r} = r - j_{\alpha} -1$ and $A(r) = B(n-r)$. Thus, we have
\begin{align*}
v_r \cdot u_{n-r} &= (-1)^{r - j_\alpha -2}x^{A(r) - B(n-r-1)}y^{B(n-r) - A(r-1)} + (-1)^{r-j_\alpha -1}x^{b_{n-r}}y^{a_r}\\
&= (-1)^{r-j_\alpha -2}x^{B(n-r) - B(n-r-1)}y^{A(r) - A(r-1)} + (-1)^{r-j_\alpha -1}x^{b_{n-r}}y^{a_r}\\
&= (-1)^{r-j_\alpha -2}x^{b_{n-r}}y^{a_r} + (-1)^{r-j_\alpha -1}x^{b_{n-r}}y^{a_r} = 0
\end{align*}
For any $2 \leq c \leq r - j_\alpha -1$, we have 
\begin{align*}
v_{j_\alpha +c} \cdot u_{n-r} &= x^{a_{j_\alpha + c}} \cdot (-1)^{j_\alpha + c - (j_\alpha + 2)}x^{A(j_\alpha + c -1) - B(n-r-1)}y^{B(n-r) - A(j_\alpha + c-1)} \\
& \qquad \qquad + y^{a_{j_\alpha + c}} \cdot (-1)^{j_\alpha + c + 1 - (j_\alpha + 2)} x^{A(j_\alpha + c) - B(n-r-1)}y^{B(n-r) - A(j_\alpha + c)}\\
&= (-1)^{c-2}x^{A(j_\alpha + c) - B(n-r-1)}y^{B(n-r) - A(j_\alpha + c-1)} \\
& \qquad \qquad + (-1)^{c-1}x^{A(j_\alpha + c) - B(n-r-1)}y^{B(n-r) - A(j_\alpha + c-1)} = 0
\end{align*}
Thus, we have $v_p \cdot u_{n-r} = 0$ for any $1 \leq p \leq r$.

We now analyze the case $a_1 \leq b_1$. Observe that for $i_1 + 1 \leq q \leq (n-r)$, the proof of the fact that $v_p \cdot u_q = 0$ for any $1 \leq p \leq r$ is exactly same as above. We only need to work out the cases $1 \leq q \leq i_1$.

By construction, the column $u_1$ has non-zero entries in rows $1,2 , \cdots, j_0 +1$ and $r+2$. The only rows of $v$ which has non-zero entry in corresponding columns are $1 \leq p \leq j_0 + 1$. We have 
\[ v_1 \cdot u_1 = x^{a_1} \cdot (-y^{b_1}) + y^{a_1} \cdot (-1)^{2}x^{A(1)}y^{B(1) - A(1)} = -x^{a_1}y^{b_1} + x^{a_1}y^{b_1} = 0 \]
For $2 \leq c \leq j_0$, we have 
\begin{align*}
v_c \cdot u_1 &= x^{a_c} \cdot (-1)^{c}x^{A(c-1)}y^{B(1) - A(c-1)} + y^{a_c} \cdot (-1)^{c+1}x^{A(c)}y^{B(1) - A(c)}\\
&= (-1)^c x^{A(c)}y^{B(1) - A(c-1)}+ (-1)^{c+1}x^{A(c)}y^{B(1) - A(c-1)}=0
\end{align*}
and lastly
\begin{align*}
v_{j_0 + 1} \cdot u_1 &= x^{a_{j_0 + 1}} \cdot (-1)^{j_0 + 1}x^{A(j_0)}y^{B(1) - A(j_0)} + x^{A(j_0 + 1) - B(1)}y^{B(1) - A(j_0)} \cdot (-1)^{\beta_1}x^{b_1}\\
&= (-1)^{j_0 + 1}x^{A(j_0 + 1)}y^{B(1) - A(j_0)} + (-1)^{j_0}x^{A(j_0 + 1)}y^{B(1) - A(j_0)} = 0
\end{align*}
Thus, $v_p \cdot u_1 = 0$ for all $1 \leq p \leq r$.

For $2 \leq q \leq i_1$, observe that $u_q$ has non-zero entry in row $r+ 1 + q-1$ and $r+1+q$. Clearly, $j_0 +1$ is the only row in $v$ with non-zero entry in the corresponding columns. We have
\begin{align*}
v_{j_0 + 1} \cdot u_q &= x^{A(j_0 + 1) - B(q-1)}y^{B(q-1) - A(j_0)} \cdot (-y^{b_q}) + x^{A(j_0 + 1) - B(q)}y^{B(q) - A(j_0)} \cdot x^{b_q}\\
&= - x^{A(j_0 + 1) - B(q-1)}y^{B(q) - A(j_0)} + x^{A(j_0 + 1) - B(q-1)}y^{B(q) - A(j_0)} = 0
\end{align*}
Thus, $v_p \cdot u_q = 0$ for all $1 \leq p \leq r$ and $2 \leq q \leq i_1$.

In conclusion, we have $v \circ u = 0$.
\end{proof}

Recall from section \ref{prelim} that $M(b_\bullet)$ is the locus of morphisms in $Mor_e(\P^1, G(r,n))$ with the restricted universal quotient bundle being isomorphic to $\mathcal{O}_{\P^1}(b_1) \oplus \cdots \oplus \mathcal{O}_{\P^1}(b_{n-r})$, and $M'(a_\bullet)$ is the locus of morphisms in $Mor_e(\P^1, G(r,n))$ with the restricted universal sub-bundle being isomorphic to $\mathcal{O}_{\P^1}(-a_1) \oplus \cdots \oplus \mathcal{O}_{\P^1}(-a_r)$. We see that 
\begin{cor}\label{corollary9}
The intersection of the loci $M(b_\bullet)$ and $M'(a_\bullet)$ is nonempty. 
\end{cor}
\begin{proof}
It follows from Proposition \ref{nonempty} that we have an exact sequence
\begin{equation}\label{eqn4}
0 \rarrow \oplus_{j=1}^{n-r} \,\mathcal{O}_{\P^1}(-b_j) \stackrel{u}{\rarrow} \mathcal{O}_{\P^1}^{\oplus n} \stackrel{v}{\rarrow} \oplus_{i=1}^r \, \mathcal{O}_{\P^1}(a_i) \rarrow 0 
\end{equation}
The surjection $v$ in equation \ref{eqn4} corresponds uniquely to an element of $Mor_e(\P^1, G(r,n))$, say $\varphi_v$. Moreover, it follows from our identification of $v$ and $\varphi_v$ in Lemma \ref{correspondence} and from equation \ref{eqn4} that $\varphi_v^*(\mathcal{S})$ is isomorphic to $\oplus_{i=1}^r \mathcal{O}_{\P^1}(-a_i)$ and $\varphi_v^*(\mathcal{Q})$ is isomorphic to $\oplus_{j=1}^{n-r} \mathcal{O}_{\P^1}(b_j)$, where $\mathcal{S}$ is the universal sub-bundle and $\mathcal{Q}$ is the universal quotient bundle of $\G$. Hence, the intersection of $M(b_\bullet)$ and $M'(a_\bullet)$ is non-empty.
\end{proof}

\section{\normalfont\textsc{The intersection locus is generically transverse}}\label{transverseintersectionsection}
In this section, we are going to show that there is a point in $M(b_\bullet) \cap M'(a_\bullet)$ where the intersection is transverse. As a consequence, we see that $M(b_\bullet)$ and $M'(a_\bullet)$ intersect generically transversely. 

More precisely, we want to show that there exists an exact sequence 
\begin{equation}\label{eqntrans} 
 0 \rarrow \oplus_{j=1}^{n-r}\,\mathcal{O}_{\P^1}(-b_j) \stackrel{u}{\rarrow} \mathcal{O}^{\oplus n} \stackrel{v}{\rarrow} \oplus_{i=1}^r \,\mathcal{O}_{\P^1}(a_i) \rarrow 0 
\end{equation}
where  $a_1 \geq \cdots \geq a_r \geq 0$, $0 \leq b_1 \leq \cdots \leq b_{n-r}$, $(n-r) \leq r$, and $a_1 + \cdots + a_r = b_1 + \cdots + b_{n-r} = e$, such that $M(b_\bullet)$ and $M'(a_\bullet)$ intersect transversely at the morphism $\varphi_v$ corresponding to the surjection $v$ (see Lemma \ref{correspondence}).

For ease of notation, let $E = \mathcal{O}(a_1) \oplus \cdots \oplus \mathcal{O}(a_r)$ and $K = \mathcal{O}(-b_1) \oplus \cdots \oplus \mathcal{O}(-b_{n-r})$. Applying $Hom (K, \bullet)$ and $Hom(\bullet, E)$ to equation \ref{eqntrans}, we obtain two long exact sequences
\begin{equation}\label{eqn}
\begin{aligned}
&0 \rarrow Hom(K,K) \rarrow Hom(K, \mathcal{O}^{\oplus n}) \rarrow Hom(K,E) \rarrow Ext^1(K,K) \rarrow 0  \\
& \text{ and } \\
&0 \rarrow Hom(E,E) \rarrow Hom(\mathcal{O}^{\oplus n}, E) \rarrow Hom(K, E) \rarrow Ext^1(E,E) \rarrow 0 
\end{aligned}
\end{equation}

We observe that 
\begin{rem}\label{remark10} To show that $M(b_\bullet)$ and $M'(a_\bullet)$ intersect transversely at $\varphi_v$, it is enough to show that the kernels of the maps $Hom(K,E) \rarrow Ext^1(K,K)$ and $Hom(K,E) \rarrow Ext^1(E,E)$ intersect transversely. 
\end{rem}

Let $W_1$ be the kernel of the map $Hom(K,E) \rarrow Ext^1(K,K)$ and $W_2$ be the kernel of the map $Hom(K,E) \rarrow Ext^1(E,E)$. Using elementary linear algebra, we deduce the following Lemma. 
\begin{lem}\label{lemma11}
The subspaces $W_1$ and $W_2$ of $Hom(K,E)$ intersect transversely iff they span $Hom(K,E)$.
\end{lem}
\begin{proof}
Note that $W_1$ and $W_2$ intersect transversely if and only if 
\[ codim(W_1 \subset Hom(K,E)) = codim((W_1 \cap W_2) \subset W_2) \]
Furthermore, it is a known fact that for any two subspaces $W_1$ and $W_2$, we have 
\[ codim((W_1 \cap W_2) \subset W_2) = codim(W_1 \subset (W_1 + W_2)) \]
Our assertion follows from these two equations.
\end{proof}

We infer from the exact sequences in equation \ref{eqn} that $W_1$ is the image of the map \newline $Hom(K, \mathcal{O}^{\oplus n}) \rarrow Hom(K,E)$, and $W_2$ is the image of the map $Hom(\mathcal{O}^{\oplus n}, E) \rarrow Hom(K,E)$. 

Consider the map 
\[ \Psi : Hom(K, \O^{\oplus n}) \times Hom(\O^{\oplus n},E) \rarrow Hom(K,E) \]given by $\Psi(\varphi,\psi) = \psi \circ u + v \circ \varphi$. Clearly, $W_1$ and $W_2 $ span $Hom(K,E)$ iff $\Psi$ is surjective.

Consider the bilinear map of vector spaces 
\[ \Phi : Hom(K, \mathcal{O}^{\oplus n}) \times Hom(\mathcal{O}^{\oplus n},E) \rarrow Hom(K,E) \] given by $\Phi(\varphi, \psi) = \psi \circ \varphi$. We see that $\Phi$ is a bilinear smooth map, so we can look at $D\Phi_{(u,v)}$. Identifying the tangent spaces with the original vector space, we get a map
\[ D\Phi_{(u,v)} : Hom(K, \mathcal{O}^{\oplus n}) \times Hom(\mathcal{O}^{\oplus n}, E) \rarrow Hom(K,E) \] given by $D\Phi_{(u,v)}(\varphi, \psi) = \psi \circ u + v \circ \varphi$. Therefore, we have $D\Phi_{(u,v)} = \Psi$ which yields 
\begin{lem}\label{lemma12}
The subspaces $W_1$ and $W_2$ intersect transversely iff $D\Phi_{(u,v)}$ is surjective.
\end{lem}

We want to show that there exists a pair $(u,v)$ with $u$ injective, $v$ surjective, $v \circ u = 0$, and $D\Phi_{(u,v)}$ is surjective. Before we proceed to show this, we make a couple of observations.

\begin{propn}
The map $\Phi$ is surjective.
\end{propn}
\begin{proof}
Let $P = (P_{i,j})_{r \times (n-r)}$ be an element of $Hom(K,E)$. We need to find elements $A \in Hom(K, \O^{\oplus n})$ and $B \in Hom(\O^{\oplus n},E)$ such that $P = A \circ B$. Clearly, $P_{i,j}$ is a homogeneous element of degree $a_i + b_j$ and hence, there exists homogeneous polynomials $R_{i,j}$ of degree $b_j$ and $Q_{i,j}$ of degree $a_i$ such that 
\[ P_{i,j} = x^{a_i} \cdot R_{i,j} + Q_{i,j} \cdot y^{b_j}\]

Consider the matrix $A = (A_{i,j})_{r \times n}$ and $B = (B_{i,j})_{n \times (n-r)}$ defined as follows : 
\[
A_{i,j} = \begin{cases}
x^{a_i},  \text{ if } i = j\\
Q_{i, j - r},  \text{ if } r+1 \leq j \leq n\\
0,  \text{ otherwise }
\end{cases} \qquad \qquad
B_{i,j} = \begin{cases}
R_{i,j}, & \text{ if } 1 \leq i \leq r \\
y^{b_j}, & \text{ if } i = r+j \\
0 , & \text{ otherwise }
\end{cases}
\]
Let $A_i$ denote the $i$th row of $A$ and $B_j$ denote the $j$th column of $B$. It follows from construction that $A_i \cdot B_j = x^{a_i}R_{i,j} + Q_{i,j}y^{b_j} = P_{i,j}$. Hence, $\Psi(A,B) = A \circ B = P$.
\end{proof}

\begin{propn}\label{balanced} When $K$ or $E$ is balanced, then $D\Phi_{(u,v)}$ is surjective.
\end{propn}
\begin{proof}
Let $K = \mathcal{O}(-b_1) \oplus \cdots \oplus \mathcal{O}(-b_{n-r}) $ is balanced. Then, we have 
\[ Ext^{1}(K,K) = H^1(\mathbb{P}^1, K^* \otimes K) = H^0(\mathbb{P}^1, K^* \otimes K \otimes \mathcal{O}(-2))^* \quad \text{ by Serre's duality} \]
Clearly,
\[ K^* \otimes K \otimes O(-2) = \oplus_{i,j} \mathcal{O}(b_i - b_j -2) \]
Since $K$ is balanced, $b_i - b_j - 2 < 0$ for all $1 \leq i,j \leq n-r$. Hence, $Ext^1(K,K) = 0$. It follows from exact sequence stated earlier (see equation \ref{eqn}) that the map 
\[ Hom(K, \mathcal{O}^{\oplus n}) \rarrow Hom(K,E) \] is surjective, and hence the map
\[ D\Phi_{(u,v)} : Hom(K, \mathcal{O}^{\oplus n}) \times Hom(\mathcal{O}^{\oplus n},E) \rarrow Hom(K,E) \] is also surjective. 

We argue similarly when $E$ is balanced.
\end{proof}

We now proceed to show that there exists a pair $(u,v)$ with $D\Phi_{(u,v)}$ is surjective. Before tackling the general case, we look at special case when $r = n-r = 2$. 

\begin{propn} When $n = 4$ and $r = 2$, then there exits a pair $(u,v)$ with $u$ injective, $v$ surjective, $v \circ u = 0$, and $D \Phi_{(u,v)}$ is surjective.
\end{propn}
\begin{proof}
Recall that in Proposition \autoref{baby}, we constructed a pair $(u,v)$ with $v$ surjective, $u$ injective, and $v \circ u = 0$. Let $P$ be an element of $Hom(K,E)$. We can think of $P$ as a $2 \times 2$ matrix $P = (P_{i,j}) $ whose $(i,j)$th entry $P_{i,j}$ is a homogeneous polynomial of degree $a_i + b_j$. 

We need to find a $4 \times 2$ matrix $R = (R_{i,j})$ and a $2 \times 4$ matrix $Q = (Q_{i,j})$, where $R_{i,j}$ has degree $b_j$ and $Q_{i,j}$ has degree $a_i$, which satisfies the equation 
\[ P = v \circ R + Q \circ u \]
Comparing the entries of the matrices, we get the following equations
\begin{align*}
&P_{1,1} = x^{a_1}R_{1,1} + y^{a_1}R_{2,1} + x^{a_1 - b_1}y^{b_1}R_{4,1} - Q_{1,1}y^{b_1} + Q_{1,4}x^{b_1} \\
&P_{1,2} = x^{a_1}R_{1,2} + y^{a_1}R_{2,2} + x^{a_1 - b_1}y^{b_1}R_{4,2} + Q_{1,2}x^{a_1 - b_1}y^{a_2} - Q_{1,3}x^{b_2} - Q_{1,4}y^{b_2} \\
&P_{2,1} = x^{a_2}R_{2,1} + y^{a_2}R_{3,1} - Q_{2,1}y^{b_1} + Q_{2,4}x^{b_1}\\
&P_{2,2} = x^{a_2} R_{2,2} + y^{a_2}R_{3,2} + Q_{2,2}x^{a_1 - b_1}y^{a_2} - Q_{2,3}x^{b_2} - Q_{2,4}y^{b_2}
\end{align*}
We solve these equations from bottom to top. First, set $R_{3,2}, Q_{2,2}, Q_{2,3}$ to be zero, and solve $R_{2,2}, Q_{2,4}$ for the equation $P_{2,2} = x^{a_2}R_{2,2} - Q_{2,4}y^{b_2}$. Then, set $R_{3,1} = 0$, and solve for $R_{2,1}, Q_{2,1}$ in the equation $P_{2,1} - Q_{2,4}x^{b_1} = x^{a_2}R_{2,1} - Q_{2,1}y^{b_1}$. Then, set $R_{4,2}, Q_{1,2}, Q_{1,3}$ to be zero, and solve for $R_{1,2},Q_{1,4}$ in the equation $P_{2,1} - y^{a_1}R_{2,2} = x^{a_1}R_{1,2} - Q_{1,4}y^{b_2}$. Finally, set $R_{4,1}=0$ and solve $R_{1,1}, Q_{1,1}$ in the equation $P_{1,1} - y^{a_1}R_{2,1} - Q_{1,4}x^{b_1} = x^{a_1}R_{1,1} - Q_{1,1}y^{b_1}$.

This shows that the map $D\Phi_{(u,v)}$ is surjective.
\end{proof}

We now proceed to the general case. 
\begin{propn}\label{gentrans}
Given any $n \geq 4$ and $2 \leq r \leq n-2$ satisfying $(n-r) \leq r$, there exists a pair $(u,v)$ with $u$ injective, $v$ surjective, $v \circ u = 0$, and $D\Phi_{(u,v)}$ is surjective.
\end{propn}
\begin{proof}
Recall that we constructed matrices $v$ and $u$ in the paragraphs preceding Proposition \autoref{nonempty}, and proved that $v$ is surjective, $u$ is injective, and $v \circ u = 0$. We just need to show that $D\Phi_{(u,v)}$ is surjective for this pair $(u,v)$. 

Let $P$ be an element of $Hom(K,E)$. We can think of $P$ as $(P_{i,j})$ which is a $r \times (n-r)$ matrix with $P_{i,j}$ being a homogeneous polynomial of degree $a_i + b_j$. We need to show that there exits elements $R \in Hom(K, \O^{\oplus n})$ and $Q \in Hom(\O^{\oplus n}, E) $ such that $P = v \circ R + Q \circ u$. We can think of $R$ as $(R_{i,j})$ which is a $(n-r) \times n$ matrix with $R_{i,j}$ being homogeneous polynomial of degree $b_j$, and $Q = (Q_{i,j})$ a $r \times n$ matrix with $Q_{i,j}$ being homogeneous polynomial of degree $a_i$. 

Observe that by comparing both sides of equation $P = v \circ R + Q \circ u$, get that for any $i,j$, we have 
\[ P_{i,j} = x^{a_i} R_{i, j} - Q_{\alpha_i, \beta_j} y^{b_j} + \text{ other terms } \]
We try to solve these equations in the following order 
\[ P_{r,n-r}, \cdots, P_{r,1}, P_{r-1, n-r}, \cdots, P_{r-1,1}, \cdots, P_{1,n-r} , \cdots, P_{1,1} \]
in the following manner : 

Assume that all equations for $P_{i,j}$ where $i > i_0$, or $i = i_0$ and $j > j_0$ are solved. As mentioned earlier we have equation 
\[ P_{i_0, j_0} = x^{a_{i_0}} R_{i_0,j_0} - Q_{\alpha_{i_0}, \beta_{j_0}} y^{b_{j_0}} + \text{ other terms } \]
where the "other terms" has a bunch of $R_{\alpha,\beta}$'s and $Q_{\alpha',\beta'}$'s occurring in them, some of which are already determined in some previous equation, and some are not. If they are not determined, then set them to be $0$. Then we solve for $R_{i_0,j_0}$ and $Q_{\alpha_{i_0}, \beta_{j_0}}  $ in the equation 
\[ P_{i_0, j_0} - \text{ other terms } = x^{a_{i_0}} R_{i_0,j_0} - Q_{\alpha_{i_0}, \beta_{j_0}} y^{b_{j_0}}  \]

We claim that we can solve for all the equations $P_{r,n-r}, \cdots, P_{1,1}$ in aforementioned method. Suppose not, consider the first $P_{i_0,j_0}$ for which a conflict occurs. Only possible conflict at this step is that $R_{i_0,j_0}$ or $Q_{\alpha_{i_0},\beta_{j_0}}$ has been already determined at some previous step. But this is not possible, because by construction of the matrices $u$ and $v$, we have that in each column of $v$ in which $x^{a_i}$ appears, all the entries below $x^{a_i}$ in that column are $0$; similarly, in each row of $u$ in which $-y^{b_j}$ appears, all the entries to the right of $-y^{b_j}$ in that row are $0$; and hence, $R_{i_0, j_0}$ and $Q_{\alpha_{i_0}, \beta_{j_0}}$ does not appear in any of the previous equations.
\end{proof}

As a corollary, we get 
\begin{cor}\label{corollary17}
There exists an exact sequence 
\[ 0 \rarrow \oplus_{j=1}^{n-r}\,\mathcal{O}_{\P^1}(-b_j) \stackrel{u}{\rarrow} \mathcal{O}^{\oplus n} \stackrel{v}{\rarrow} \oplus_{i=1}^r \,\mathcal{O}_{\P^1}(a_i) \rarrow 0 \]
such that the loci $M(b_\bullet)$ and $M'(a_\bullet)$ intersect transversely at the morphism $\varphi_v$ corresponding to the surjection $v$.

In particular, the loci $M(b_\bullet)$ and $M'(a_\bullet)$ intersect generically transversely and has an irreducible component of codimension 
\[ \sum_{i,j} \max \left\lbrace a_i - a_j - 1, 0 \right\rbrace + \sum_{i,j} \max \left\lbrace b_i - b_j - 1, 0 \right\rbrace \]

Moreover, if either of the splitting type $\{ a_\bullet \}$ or $\{ b_\bullet \}$ is balanced, then the intersection is transverse.
\end{cor}
\begin{proof}
The first assertion of the corollary follows from Remark \ref{remark10}, Lemma \ref{lemma11}, Lemma \ref{lemma12}, and  Proposition \ref{gentrans}. 

The second assertion follows from the first one and Proposition \ref{codimensionlemma}.

The third assertion follows from Remark \ref{remark10}, Lemma \ref{lemma11}, Lemma \ref{lemma12}, and Proposition \ref{balanced}.
\end{proof}

In summary, it follows from Corollary \ref{corollary9} and \ref{corollary17} that 
\begin{thm}\label{maintheorem}
The intersection of the loci $M(b_\bullet)$ and $M'(a_\bullet)$ is nonempty and generically transverse. Furthermore, if either of the splitting types $\{ a_\bullet \}$ or $\{ b_\bullet \}$ is balanced, then the intersection is transverse.
\end{thm}


\section{\normalfont\textsc{Analyzing the locus with restricted tangent bundle having fixed splitting type}}\label{section5} 
In this section, we are going to show that the locus of morphisms in $Mor_e(\P^1, \G)$ with the restricted tangent bundle having fixed splitting type need not always be irreducible. This is in sharp contrast with the results of Verdier \cite{verdier} and Ramella \cite{ramella}, who  have shown that given a collection of integers $a_1, \cdots, a_n$ with $a_1 \geq \cdots \geq a_n$ and $\sum_{i=1}^n a_i = e$, the locus of morphisms $\varphi$ in $Mor_e(\P^1, \P^n)$ with the restricted twisted tangent bundle $\varphi^*(T_{\P^n}(-1))$ having splitting type $(a_1, \cdots, a_n)$ is empty if $a_n <0$, else it is nonempty, smooth and connected of codimension 
\[ \sum_{i,j} \max \{ a_i - a_j - 1, 0 \} \]

Recall that given a morphism $\varphi : \P^1 \rarrow G(r,n)$, the restricted tangent bundle $\varphi^* (T_{G(r,n)})$ is isomorphic to $\varphi^* (\mathcal{S}^*) \otimes \varphi^*(\mathcal{Q})$, where $\mathcal{S}$ and $\mathcal{Q}$ are the universal sub-bundle and universal quotient bundle of $\G$. Now let us fix a splitting type $c_1, \cdots, c_{r(n-r)}$ for the restricted tangent bundle $\varphi^*(T_{\G})$. We define 
\begin{defn}\label{filling}
A \textit{filling} for the splitting type $\lbrace c_l \rbrace_{1 \leq l\leq r(n-r)}$ to be a $r \times (n-r)$ matrix $A$ with entries $a_{i,j} = c_l$ for some $l$ depending on $i,j$ such that 
\begin{enumerate}[$\bullet$]
\item For all $1 \leq i \leq r-1$ and $1 \leq j \leq n-r-1$, we have $a_{i,j} \leq a_{i+1,j}$ and $a_{i,j} \leq a_{i,j+1}$.
\item For all $1 \leq i \leq r-1$ we have $a_{i,n-r} \leq a_{i+1, n-r}$, and for all $1 \leq j \leq n-r-1$ we have $a_{r,j} \leq a_{r,j+1}$.
\item For all $1 \leq i \leq r-1$ the difference $a_{i+1,j} - a_{i,j}$ is independent of $j$, and for all $1 \leq j\leq n-r-1$ the difference $a_{i,j+1} - a_{i,j}$ is independent of $i$.
\end{enumerate}
\end{defn}
Moreover, we define 
\begin{defn}
A collection of integers $\alpha_1, \cdots, \alpha_\nu$ is \textit{non-negative} if $\alpha_i$ are non-negative integers for all $1 \leq i \leq \nu$. A collection of integers $\alpha_1, \cdots, \alpha_\nu$ is \textit{increasing} if $\alpha_1 \leq \cdots \leq \alpha_\nu$.
\end{defn}

The exigency of these definitions is due to the following Lemma. 
\begin{lem}\label{fillingtosplitting}
A filling for the splitting type $\lbrace c_l \rbrace_{1 \leq l \leq r(n-r)}$ uniquely determines the non-negative increasing splitting type of $\varphi^*(\mathcal{S}^*)$ and $\varphi^*(\mathcal{Q})$.
\end{lem}
\begin{proof}
Let $\varphi^*(\mathcal{S}^*)$ be isomorphic to $\oplus_{i=1}^r \,\mathcal{O}_{\P^1}(a_i)$, and let $\varphi^*(\mathcal{Q})$ be isomorphic to $\oplus_{j=1}^{n-r} \,\mathcal{O}_{\P^1}(b_j)$. We can determine the $a_i$'s and $b_j$'s uniquely by the following equations
\begin{align*}
e &= \frac{1}{n}\sum_{i,j} a_{i,j}  \\
a_i &= \frac{1}{n-r}\left( \sum_{j=1}^{n-r} a_{i,j} - e \right) \qquad \text{ for all } 1 \leq i \leq r \\
b_j &= \frac{1}{r}\left( \sum_{i=1}^r a_{i,j} - e \right) \qquad \qquad\text{ for all } 1 \leq j \leq n-r
\end{align*}

Conversely, given a splitting type $\{a_\bullet\}$ for $\varphi^*(\mathcal{S}^*)$ and $\{b_\bullet\}$ for $\varphi^*(\mathcal{Q})$ with $0 \leq a_1 \leq \cdots \leq a_r$ and $0 \leq b_1 \leq \cdots \leq b_{n-r}$, we define a filling whose $(i,j)$th entry is $a_i + b_j$.
\end{proof}

Let $\{a_\bullet \}$ and $\{a'_\bullet \}$ be two non-negative increasing splitting types for $\varphi^*(\mathcal{S}^*)$, and let $\{ b_\bullet \}$ and $\{ b'_\bullet \}$ be two non-negative increasing splitting types for $\varphi^*(\mathcal{Q})$. If $\{ a_\bullet \}$ is different from $\{a'_\bullet \}$ (i.e. the corresponding vector bundles are not isomorphic) or $\{b_\bullet\}$ is different from $\{b'_\bullet\}$, then the intersection of loci $M(b_\bullet) \cap M'(a_\bullet)$ and $M(b'_\bullet)\cap M'(a'_\bullet)$ must be empty because a morphism $\varphi : \P^1 \rarrow \G$ uniquely determines the splitting type for $\varphi^*(\mathcal{S}^*)$ and $\varphi^*(\mathcal{Q})$. Hence, it follows from Lemma \ref{fillingtosplitting} that 
\begin{propn}\label{proposition22}
The locus of morphisms $\varphi$ in $Mor_e(\P^1, G(r,n))$ with the restricted tangent bundle having the splitting type $\{c_l \}_{1\leq l \leq r(n-r)}$ is stratified by the loci $M(b_\bullet) \cap M'(a_\bullet)$ where $\{a_\bullet\}$ and $\{b_\bullet \}$ are non-negative increasing splitting types for $\varphi^*(\mathcal{S}^*)$ and $\varphi^*(\mathcal{Q})$ arising from the distinct fillings for $\{c_l \}_{1 \leq l \leq r(n-r)}$.
\end{propn}

Recall that given a collection of non-negative increasing integers $\alpha_1, \cdots, \alpha_\nu$, we defined its \textit{polygonal line} (see Definition \ref{polygon}) to be $\mathfrak{P}(\alpha_\bullet) = (\alpha_\nu, \alpha_\nu + \alpha_{\nu-1}, \cdots, \alpha_\nu + \cdots + \alpha_1)$.
\begin{defn}
We say a filling $\lbrace a_{i,j} \rbrace_{{1\leq i \leq r, \, 1 \leq j \leq n-r}}$ of a splitting type $\lbrace c_l \rbrace_{1 \leq l \leq r(n-r)}$ to be \textit{minimal} if the following holds:

Let $a_1, \cdots, a_r$ be the non-negative increasing splitting type of $\varphi^*(\mathcal{S}^*)$ and $b_1 , \cdots, b_{n-r}$ be the non-negative increasing splitting type of $\varphi^*(\mathcal{Q})$ uniquely determined by the filling (see Lemma \ref{fillingtosplitting}). Then for every possible non-negative increasing collection of integers $a'_1, \cdots, a'_r$ and $b'_1, \cdots, b'_{n-r}$ with $a'_1 + \cdots + a'_r = b'_1 + \cdots + b'_{n-r} = e$ satisfying $\mathfrak{P}(a'_\bullet) \geq \mathfrak{P}(a_\bullet)$ and $\mathfrak{P}(b'_\bullet) \geq  \mathfrak{P}(b_\bullet)$ with atleast one of the inequality being strict, the matrix $\lbrace a'_i + b'_j \rbrace_{{1\leq i \leq r, \, 1 \leq j \leq n-r}}$ is not a filling for $\lbrace c_l \rbrace$.
\end{defn}

It follows as a consequence of Corollary \ref{mclosure2} that 
\begin{lem}\label{lemmamclosure}Suppose $\lbrace a_\bullet \rbrace$ and $\lbrace b_\bullet \rbrace$ be the non-negative increasing splitting types for $\varphi^*(\mathcal{S})^*$ and $\varphi^*(\mathcal{Q})$ respectively arising from a minimal filling of a given splitting $\lbrace c_l \rbrace$, then the loci $M(b_\bullet) \cap M'(a_\bullet)$ is closed in the locus of all degree $e$ morphisms with restricted tangent bundle having splitting type $\lbrace c_l \rbrace$.
\end{lem}

Hence, we see that
\begin{cor}\label{corollary23}
The number of irreducible components of the locus of degree $e$ morphisms from $\mathbb{P}^1$ to $G(r,n)$ with the restricted tangent bundle having a given splitting type is bounded below by the number of distinct minimal fillings of the given splitting type. In particular, this locus need not always be irreducible.
\end{cor}
\begin{proof}
The proof follows from Proposition \ref{proposition22} and Lemma \ref{lemmamclosure}, Corollary \ref{corollary9} and \ref{corollary17}.
\end{proof}

For example, let $r =2$, $n=4$ and $e=6$. The locus $Mor_6(\mathbb{P}^1, G(2,4))$ of degree $6$ morphisms from $\mathbb{P}^1$ to $G(2,4)$ has dimension $28$. Consider the splitting type $3,5,7,9$ for the restricted tangent bundle. We have two possible fillings 
\[ \begin{pmatrix}
3 & 5 \\ 7 & 9
\end{pmatrix}
\qquad \text{ and } \qquad 
\begin{pmatrix}
3 & 7 \\ 5 & 9
\end{pmatrix}
\]
Corresponding to the first filling we have non-negative increasing splitting types $(a_1, a_2) = (1,5)$ and $(b_1,b_2) = (2,4)$, and to the second filling we have $(a_1,a_2) = (2,4)$ and $(b_1,b_2) = (1,5)$. Since both the fillings are minimal (see Lemma \ref{lemma24}), the locus of morphisms in $Mor_6(\mathbb{P}^1, G(2,4))$ is the disjoint union of the loci $M(2,4) \cap M'(1,5)$ and $M(1,5) \cap M'(2,4)$. Now the loci $M(1,5)$ and $M'(1,5)$ have codimension $3$ in $Mor_6(\mathbb{P}^1, G(2,4))$ which follows from Proposition \ref{codimensionlemma}. Similarly, the loci $M(2,4)$ and $M'(2,4)$ have codimension $1$ in $Mor_6(\mathbb{P}^1, G(2,4))$. It follows from Corollary \ref{corollary9} that the locus $M(2,4) \cap M'(1,5)$ is nonempty. Similarly, $M(1,5) \cap M'(1,5)$ is also nonempty, and since $M(2,4)$ and $M(1,5)$ are disjoint, the intersection $M(2,4) \cap M'(1,5)$ must be proper subset of $M'(1,5)$. Moreover, since $M(2,4)$ has codimension $1$ in $Mor_6(\mathbb{P}^1, G(2,4))$, the intersection locus $M(2,4) \cap M'(1,5)$ must have codimension $1$ in $M'(1,5)$, and hence, it must have codimension $4$ in $Mor_6(\mathbb{P}^1, G(2,4))$. Similarly, the intersection locus $M(1,5) \cap M'(2,4)$ has codimension $4$ in $Mor_6(\mathbb{P}^1, G(2,4))$. Hence, the locus of degree $6$ morphisms from $\mathbb{P}^1$ to $G(2,4)$ with restricted tangent bundle having splitting type $3,5,7,9$ has codimension $4$ and has atleast two irreducible components arising from the two distinct fillings.

We see that Proposition \ref{proposition22} exhorts us to determine the possible fillings of a  splitting type as a key step towards understanding the locus of morphisms in $Mor_e(\P^1, \G)$ with restricted tangent bundles having the given splitting type. To this end, we have the following Lemmas.
\begin{lem}\label{lemma24}
Let $r = 2 $ and $n=4$, and let $\{ c_1, c_2, c_3, c_4 \}$ be a splitting type of the restricted tangent bundle with $c_1 \leq c_2 < c_3 \leq c_4$. Then $\lbrace c_1, c_2, c_3, c_4 \rbrace$ has two possible fillings 
\[ \begin{pmatrix}
c_1 & c_2 \\ c_3 & c_4
\end{pmatrix} \qquad \text{ and } \qquad 
\begin{pmatrix}
c_1 & c_3 \\
c_2 & c_4
\end{pmatrix} \]
Moreover, both the fillings are minimal.
\end{lem}
Similarly, we have 
\begin{lem}\label{lemma25}
Let $r=3$ and $n=5$. A splitting type $\lbrace c_1, \cdots, c_6 \rbrace$ of the restricted tangent bundle with $c_1 \leq \cdots \leq c_6$ has exactly one filling except when $\lbrace c_1, \cdots, c_6 \rbrace = \lbrace c_1, c_1 + \lambda, \cdots, c_1 + 5 \lambda \rbrace$ for some integer $\lambda$ in which case there are two possible fillings 
\[ 
\begin{pmatrix}
c_1 & c_1 + \lambda \\
c_1 + 2 \lambda & c_1 + 3 \lambda \\
c_1 + 4\lambda & c_1 + 5\lambda
\end{pmatrix} 
\qquad \text{ and } \qquad 
\begin{pmatrix}
c_1 & c_1 + 3 \lambda \\
c_1 + \lambda & c_1 + 4\lambda\\
c_1 + 2 \lambda & c_1 + 5 \lambda
\end{pmatrix}
\]
Additionally, all the fillings are minimal.
\end{lem}
\begin{proof}[Proof of Lemma \ref{lemma24} and \ref{lemma25}]
We will briefly sketch the proof of Lemma \ref{lemma25}. One can prove Lemma \ref{lemma24} in a similar fashion. 

Let $r = 3$ and $n=5$. Given a splitting type $\{c_1, \cdots, c_6\}$ of a restricted tangent bundle $\varphi^*(T_{\G})$, there is at least one filling (since $\varphi^*(T_{\G}) = \varphi^*(\mathcal{S}^*) \otimes \varphi^*(\mathcal{Q})$), say $A$, which is a $3 \times 2$ matrix. After subtracting the $(1,1)$th entry from every other entry of $A$, we get a new matrix of form 
\[ \begin{pmatrix} 0 & \lambda \\ \rho_2 & \rho_2 + \lambda \\ \rho_3 & \rho_3 + \lambda \end{pmatrix} \]
for some non-negative integers $\lambda, \, \rho_2, \, \rho_3$ with $\rho_2 \leq \rho_3$. We now look at every possible permutations with the $(1,1)$th entry being zero and the $(3,2)$th entry being $\rho_2 + \lambda$, and force the conditions of definition \ref{filling} which gives us some equations which must be compatible. This gives us all the possibilities. A similar brute force method works for $r=2$ and $n=4$.

The proof of minimality of the fillings in the cases $r=2,\,n=4$ and $r=3,\,n=5$ are in a similar flavor. The key idea is to use the fact that the $(1,1)$th and $(r,n-r)$th entries are the same for every possible filling. For instance, when $r=3,\,n=5$, let $\lbrace c_1, \cdots, c_6 \rbrace = \lbrace c_1, c_1 + \lambda, \cdots , c_1 + 5 \lambda \rbrace$ and let $\lbrace a_\bullet \rbrace$ and $\lbrace b_\bullet \rbrace$ be the corresponding induced splittings. For any $\lbrace a'_\bullet \rbrace$ with $\mathfrak{P}(a'_\bullet) \geq \mathfrak{P}(a_\bullet)$ and $\lbrace b'_\bullet \rbrace$ with $\mathfrak{P}(b'_\bullet) \geq \mathfrak{P}(b_\bullet)$, we must have $a'_3 + b'_2 = a_3 + b_2$ and since $a'_3 \geq a_3$ and $b'_2 \geq b_2$, we get $a'_3 = a_3, \, b'_2 = b_2$. This gives $b'_1 = b_1$ and $a'_2 + a'_1 = a_2 + a_1 $. Since we must have $a'_1 + b'_1 = a_1 + b_1$, we get $a'_i = a_i$ for all $i = 1,\,2,\,3$ and $b'_j = b_j$ for all $j = 1,\,2$.
\end{proof}
Using a similar method as in proof of Lemma \ref{lemma25}, we deduce that when $r=4$ and $n=6$, a splitting type of the restricted tangent bundle of the form $\{ c_1, c_1, c_2, c_2, c_3, c_3, c_4,c_4 \}$ with $0 \leq c_2 - c_1 = c_3 - c_2  = c_4 - c_3 $ has three possible fillings 
\[ \begin{pmatrix}
c_1 & c_1 \\ c_2 & c_2 \\ c_3 & c_3 \\ c_4 & c_4 
\end{pmatrix}
\, , \, 
\begin{pmatrix}
c_1 & c_2 \\ c_1 & c_2 \\ c_3 & c_4 \\ c_3 & c_4 
\end{pmatrix}
\, \text{ and } \,
\begin{pmatrix}
c_1 & c_3 \\ c_1 & c_3 \\ c_2 & c_4 \\ c_2 & c_4
\end{pmatrix}
\]
However, in general, we found it impossible to determine all possible fillings using this brute force method. 

Additionally, we observe from these special cases that the number of fillings seems to increase as we increase $r,n$ and $e$. We don't know how the fillings of a given splitting type depend on $r$ and $n$, but we can provide a very crude upper bound for the number of possible fillings. 
\begin{lem}
The total number of distinct fillings of a splitting type $\{ c_l \}_{1 \leq l \leq r(n-r)}$ of the restricted tangent bundle is bounded above by $\binom{r(n-r)-2}{n-r-1}$.
\end{lem}
\begin{proof}
It follows from definition \ref{filling} that every filling must have the same $(1,1)$th and $(r,n-r)$th entry. Furthermore, we see that every filling is uniquely determined by the entries $(1,2),$ $\cdots,$ $(1,n-r)$. Hence, a clumsy upper bound for the total number of fillings is the number of choices for these entries, which is $\binom{r(n-r) -2}{n-r-1}$.
\end{proof}

On a more positive note, we see that 
\begin{lem}
If the splitting type of the restricted tangent bundle $\varphi^*(T_{\G})$ is balanced, then the splitting type of the restricted universal sub-bundle $\varphi^*(\mathcal{S})$ and the splitting type of the restricted universal quotient bundle $\varphi^*(\mathcal{Q})$ must be balanced.
\end{lem}
\begin{proof}
Let us choose a filling for the splitting type of the restricted tangent bundle, and let $a_1, \cdots, a_r$ and $b_1, \cdots, b_{n-r}$ be non-negative increasing splitting types of $\varphi^*(\mathcal{S}^*)$ and $\varphi^*(\mathcal{Q})$ respectively. Since the splitting type of the restricted tangent bundle is balanced, we must have $(a_r + b_{n-r}) - (a_1 + b_1) \leq 1$, which yields $a_r - a_1 \leq 1$ and $b_{n-r} - b_1 \leq 1$. Hence, the splitting types of $\varphi^*(\mathcal{S})$ and $\varphi^*(\mathcal{Q})$ must be balanced.
\end{proof}

In conclusion, the locus of morphisms in $Mor_e(\P^1, \G)$ need not always be irreducible. For example, when $r =2 $ and $n=4$, and let $c_1, c_2 , c_3, c_4$ be non-negative increasing splitting type of the restricted universal tangent bundle, with $c_2 < c_3$. It follows from   Lemma \ref{lemma24} that this locus has at least two irreducible components. 

{\small
\renewcommand{\refname}{\textsc{References}}

{\footnotesize\textsc{Department of Mathematics, Statistics and CS, University of Illinois at Chicago, Chicago, IL, 60607.}\\
\textit{E-mail address}: \texttt{smanda9@uic.edu}}}
\end{document}